\documentclass[12pt,oneside,english]{amsart}
\usepackage{lmodern}

\usepackage[T1]{fontenc}
\usepackage[latin9]{inputenc}
\usepackage{geometry}
\usepackage{babel}
\usepackage{amstext}
\usepackage{amsthm}
\usepackage{amssymb}
\usepackage{setspace}
\usepackage{microtype}
\usepackage[unicode=true,pdfusetitle,
 bookmarks=true,bookmarksnumbered=false,bookmarksopen=false,
 breaklinks=false,pdfborder={0 0 1},backref=false,colorlinks=false]
 {hyperref}

\usepackage[nameinlink,capitalize]{cleveref}

\makeatletter

\numberwithin{equation}{section}
\numberwithin{figure}{section}
\theoremstyle{plain}
\newtheorem*{question*}{\protect\questionname}
\theoremstyle{plain}
\newtheorem{thm}{\protect\theoremname}[section]
\theoremstyle{plain}
\newtheorem{fact}[thm]{\protect\factname}
\theoremstyle{plain}
\newtheorem{cor}[thm]{\protect\corollaryname}
\theoremstyle{plain}
\newtheorem{conjecture}[thm]{\protect\conjecturename}
\theoremstyle{definition}
\newtheorem{defn}[thm]{\protect\definitionname}
\theoremstyle{remark}
\newtheorem{rem}[thm]{\protect\remarkname}
\theoremstyle{plain}
\newtheorem{lem}[thm]{\protect\lemmaname}
\theoremstyle{plain}
\newtheorem{prop}[thm]{\protect\propositionname}
\theoremstyle{plain}
\newtheorem{obs}[thm]{Observation}

\@ifundefined{date}{}{\date{}}
\usepackage{amsmath}

\usepackage{autobreak}

\usepackage{xcolor}
\definecolor{blue}{RGB}{14,107,217}
\definecolor{green}{RGB}{0,158,40}
\definecolor{red}{RGB}{235,16,16}
\definecolor{brown}{RGB}{164,66,0}
\definecolor{orange}{RGB}{231,135,26}
\definecolor{purple}{RGB}{94,53,177}

\usepackage{amssymb}

\makeatother

\providecommand{\conjecturename}{Conjecture}
\providecommand{\corollaryname}{Corollary}
\providecommand{\definitionname}{Definition}
\providecommand{\factname}{Fact}
\providecommand{\lemmaname}{Lemma}
\providecommand{\propositionname}{Proposition}
\providecommand{\questionname}{Question}
\providecommand{\remarkname}{Remark}
\providecommand{\theoremname}{Theorem}

\begin{document}
\global\long\def\N{\mathbb{N}}%
\global\long\def\Z{\mathbb{Z}}%
\global\long\def\Q{\mathbb{Q}}%
\global\long\def\R{\mathbb{R}}%
\global\long\def\C{\mathbb{C}}%
 
\global\long\def\cM{\mathcal{\mathcal{M}}}%
\global\long\def\cN{\mathcal{N}}%
\global\long\def\cZ{\mathcal{Z}}%
\global\long\def\cQ{\mathcal{Q}}%
 
\title{On dp-minimal expansions of the integers}
\author{Eran Alouf}
\address{Einstein Institute of Mathematics, Hebrew University of Jerusalem,
91904, Jerusalem Israel.}
\email{Eran.Alouf@mail.huji.ac.il}
\begin{abstract}
We show that if $\mathcal{Z}$ is a dp-minimal expansion of $\left(\mathbb{Z},+,0,1\right)$
that defines an infinite subset of $\mathbb{N}$, then $\mathcal{Z}$
is interdefinable with $\left(\mathbb{Z},+,0,1,<\right)$. As a corollary,
we show the same for dp-minimal expansions of $\left(\mathbb{Z},+,0,1\right)$
which do not eliminate $\exists^{\infty}$.
\end{abstract}

\maketitle

\section{Introduction}

For two structures $\cM_{1}$,$\cM_{2}$ with the same underlying
universe $M$, we say that $\cM_{1}$ is a \emph{reduct} of $\cM_{2}$,
and that $\cM_{2}$ is an \emph{expansion} of $\cM_{1}$, if for every
$k\in\N$, every subset of $M^{k}$ which is definable in $\cM_{1}$
is also definable (with parameters) in $\cM_{2}$. We say that $\cM_{1}$
and $\cM_{2}$ are \emph{interdefinable} if $\cM_{1}$ is a reduct
of $\cM_{2}$ and $\cM_{2}$ is a reduct of $\cM_{1}$, and we say
that $\cM_{2}$ is a \emph{proper expansion} of $\cM_{1}$ if $\cM_{2}$
is an expansion of $\cM_{1}$ but $\cM_{1}$ and $\cM_{2}$ are not
interdefinable.

The dp-rank of a theory is a notion of rank that measures how much
a theory is NIP. Having dp-rank infinity (equivalently, at least $\left\vert T\right\vert ^{+}$)
is equivalent to having IP. The dp-minimal theories --- theories
with dp-rank $1$ --- are the simplest among the NIP theories, yet
the notion of dp-minimality still generalizes other notions of minimal
NIP theories, such as o-minimality and C-minimality. Often a general
study of NIP would start with studying the dp-minimal case. For example,
UDTFS was first shown for dp-minimal theories \cite{Gui12}, and later
generalized to all NIP theories \cite{CS15} (a local version was
recently proved in \cite{EK19}). Closer to our interest is the classification
of dp-minimal fields \cite{Joh15}, which was recently generalized
in \cite{Joh20} to dp-finite fields (i.e., fields of finite dp-rank).

In many cases, dp-minimality has strong global consequences for definable
sets. For example, in \cite{Sim11} it was shown that if $G$ is a
definably complete expansion of a divisible ordered group, then $G$
is dp-minimal if and only if it is o-minimal. In particular, an expansion
of $\left(\R,+,0,<\right)$ is dp-minimal if and only if it is o-minimal.
Similar results for other structures were proved in \cite{SW19}.
In \cite{SW19a}, general results were proved about definable sets
in dp-minimal structures equipped with definable uniformities satisfying
certain assumptions. These results generalize previous results about
weakly o-minimal, C-minimal and P-minimal theories. 

In this paper, we give a partial answer to the following question:
\begin{question*}
What are the dp-minimal expansions of $\left(\Z,+,0,1\right)$?
\end{question*}
Until recently, $\left(\Z,+,0,1,<\right)$ was the only known expansion
of $\left(\Z,+,0,1\right)$ of finite dp-rank. A number of results
were proved about the inexistence of such expansions. In \cite{Aschenbrenner_et_al_I_2015}
it was shown that $\left(\Z,+,0,1,<\right)$ has no proper dp-minimal
expansions. This was later significantly strengthened in \cite{DolichGoodrick2017}
by the following:
\begin{fact}[{\cite[Corollary 2.20]{DolichGoodrick2017}}]
\label{fact:no_strongly_dependent_expansions_of_the_order}Suppose
that $\cZ$ is a strong expansion of $\left(\mathbb{Z},+,0,1,<\right)$.
Then $\mathcal{Z}$ is interdefinable with $\left(\mathbb{Z},+,0,1,<\right)$.
\end{fact}

Since every strongly-dependent theory is strong, this fact implies
that every proper expansion of $\left(\mathbb{Z},+,0,1,<\right)$
has $\mbox{dp-rank}\ge\omega$. In \cite{Conant2018_no_intermediate}
the following was proved:
\begin{fact}[\cite{Conant2018_no_intermediate}]
\label{fact:no_intermediate_structures_between_the_group_and_the_order}Suppose
that $\cZ$ is an expansion of $\left(\mathbb{Z},+,0,1\right)$ and
a reduct of $\left(\mathbb{Z},+,0,1,<\right)$. Then $\mathcal{Z}$
is interdefinable with either $\left(\mathbb{Z},+,0,1\right)$ or
$\left(\Z,+,0,1,<\right)$.
\end{fact}

Together, this means that any dp-minimal (or even strong) expansion
of $\left(\mathbb{Z},+,0,1\right)$ which is not interdefinable with
$\left(\mathbb{Z},+,0,1,<\right)$, cannot define any set that is
definable in $\left(\mathbb{Z},+,0,1,<\right)$ but not in $\left(\mathbb{Z},+,0,1\right)$.
In \cite{ConantPillay2018} it was shown that $\left(\Z,+,0,1\right)$
has no proper stable expansions of finite dp-rank. In contrast, there
are superstable expansions of $\left(\mathbb{Z},+,0,1\right)$ which
are not dp-finite \cite{Conant2019_stability_sparsity} \cite{LambottePoint2017}.
For example, the structure $\left(\Z,+,0,1,\Pi_{q}\right)$, where
$2\le q\in\N$ and $\Pi_{q}:=\left\{ q^{n}\,:\,n\in\N\right\} $,
is superstable (hence strongly-dependent) of $U$-rank $\omega$.
This was shown independently in \cite{Poizat_2014} and \cite{PS18},
and the stability of this structure was also shown in \cite{MS04}.
By the result of \cite{ConantPillay2018} mentioned above, this structure
is not dp-finite.

There are simple unstable expansions of $\left(\Z,+,0,1\right)$ of
finite inp-rank. In \cite{KS16} it was shown that, under Dickson's
conjecture about the distribution of primes in the natural numbers,
the structure $\left(\Z,+,0,1,\text{Pr}\right)$ is supersimple of
$U$-rank $1$, where $\text{Pr}$ is the set of primes and their
negations. In \cite{BT17} it was shown (without the assumption of
any conjecture) that the structure $\left(\Z,+,0,1,\text{Sqf}\right)$
is supersimple of $U$-rank $1$, where $\text{Sqf}:=\left\{ a\in\Z\,:\,\text{for all primes }p\text{, }p^{2}\nmid a\right\} $.
Since $U$-rank $1$ implies inp-minimality, both $\left(\Z,+,0,1,\text{Pr}\right)$
and $\left(\Z,+,0,1,\text{Sqf}\right)$ are inp-minimal.

In \cite[Question 5.32]{Aschenbrenner_et_al_II_2013}, the authors
ask whether every dp-minimal expansion of $\left(\Z,+,0,1\right)$
is a reduct of $\left(\Z,+,0,1,<\right)$ (which, by \cref{fact:no_intermediate_structures_between_the_group_and_the_order},
must be interdefinable with either $\left(\mathbb{Z},+,0,1\right)$
or $\left(\Z,+,0,1,<\right)$). In \cite{AD19}, a new family of dp-minimal
expansions of $\left(\Z,+,0,1\right)$ was introduced, thus giving
a negative answer to this question. More generally, it was shown that
for every nonempty (possibly infinite) set of primes $P\subseteq\mathbb{N}$,
the structure $\left(\Z,+,0,1,\left\{ |_{p}\right\} _{p\in P}\right)$
has dp-rank $\left\vert P\right\vert $, where $a|_{p}b$ is interpreted
as $v_{p}(a)\leq v_{p}(b)$, with $v_{p}$ the $p$-adic valuation
on $\Z$. In particular, for a single prime $p$, the structure $\left(\Z,+,0,1,|_{p}\right)$
is dp-minimal. This was recently generalized in \cite{Cla20}, where
the following was shown: let $\left(B_{i}\right)_{i<\omega}$ be a
strictly descending chain of subgroups of $\Z$, $B_{0}=\Z$, and
let $v:\Z\to\omega\cup\left\{ \infty\right\} $ be the valuation defined
by $v(x):=\max\left\{ i\,:\,x\in B_{i}\right\} $. Then $\left(\Z,+,0,1,v\right)$,
as well as certain expansions of it, are dp-minimal. 

Another new family of dp-minimal expansions of $\left(\Z,+,0,1\right)$
was introduced in \cite{TW17}. They showed that for every $\alpha\in\R\backslash\Q$,
the structure $\left(\Z,+,0,1,C_{\alpha}\right)$ is dp-minimal, where
$C_{\alpha}$ is the ternary relation given by $C_{\alpha}(j,k,l)\iff C(\alpha j+\Z,\alpha k+\Z,\alpha l+\Z)$,
where $C$ is the usual positively oriented cyclic order on $\R/\Z$.
They also showed that $\left(\Z,+,0,1,C_{\alpha}\right)$ and $\left(\Z,+,0,1,C_{\beta}\right)$
are interdefinable if and only if $\alpha$ and $\beta$ are linearly-dependent
over $\Q$. In particular, the family of structures $\left\{ \left(\Z,+,0,1,C_{\alpha}\right)\right\} _{\alpha\in\R\backslash\Q}$
contains a subset of size $2^{\aleph_{0}}$ such that no two distinct
members of this subset are interdefinable.

These new families have something in common that sets them apart from
$\left(\Z,+,0,1,<\right)$: By the quantifier elimination result in
\cite{AD19}, it follows that all the structures of the form $\left(\Z,+,0,1,|_{p}\right)$
(as above) eliminate $\exists^{\infty}$ (in this paper, when we mention
elimination of $\exists^{\infty}$, we always mean elimination of
$\exists^{\infty}$ in the home sort). It also follows that every
subset of $\Z$ which is definable in $\left(\Z,+,0,1,|_{p}\right)$
is a finite union of sets, each of which is either a singleton or
of the form $m+n\Z$. Hence, $\left(\Z,+,0,1,|_{p}\right)$ does not
define an infinite subset of $\N$. Likewise, by \cite[Theorem 3.4]{TW17},
it follows that all the structures of the form $\left(\Z,+,0,1,C_{\alpha}\right)$
(as above) do not define an infinite subset of $\N$. Additionally,
if $G$ is a dp-minimal expansion of a group in which there is a uniformly
definable family of subsets of $G$ which forms a neighborhood basis
at the identity for a non-discrete group topology on $G$, then $G$
eliminates $\exists^{\infty}$. This immediately follows from \cite[Lemma 3.1]{JSW17}.
Since $\left(\Z,+,0,1,C_{\alpha}\right)$ defines such a group topology,
it eliminates $\exists^{\infty}$. Note that $\left(\Z,+,0,1,|_{p}\right)$
also defines such a group topology, so this gives an alternative way
to see that $\left(\Z,+,0,1,|_{p}\right)$ eliminates $\exists^{\infty}$.

We show that this is not a coincidence. We prove the following theorem:
\begin{thm}
\label{thm:new-dp-minimal-infinite-set-of-natural-numbers}Let $\mathcal{Z}$
be a dp-minimal expansion of $\left(\mathbb{Z},+,0,1\right)$, and
suppose that there exists an infinite set $A\subseteq\mathbb{N}$
that is definable in $\mathcal{Z}$. Then $\mathcal{Z}$ is interdefinable
with $\left(\mathbb{Z},+,0,1,<\right)$.
\end{thm}

As a corollary, we also prove:
\begin{thm}
\label{thm:dp-minimal-which-does-not-eliminate-exists-infty}Let $\mathcal{Z}$
be a dp-minimal expansion of $\left(\mathbb{Z},+,0,1\right)$ which
does not eliminate $\exists^{\infty}$. Then $\mathcal{Z}$ is interdefinable
with $\left(\mathbb{Z},+,0,1,<\right)$.
\end{thm}

Let $\cZ$ be an unstable expansion of $\left(\mathbb{Z},+,0,1\right)$,
and let $\cM$ be a highly saturated elementary extension of $\cZ$
with domain $M$. Then there is a unary subset of $M$ which is definable
in $\cM$ but not definable in $\left(M,+,0,1\right)$. This follows
from the fact that instability is witnessed by a formula $\phi(x,y)$
over $\emptyset$ with $\left\vert x\right\vert =1$ (see \cite[Lemma 5.13]{AD19}).
In contrast, we have the following corollary:
\begin{cor}
\label{cor:cor_of_both_thms_together}Let $\mathcal{Z}$ be a dp-minimal
expansion of $\left(\mathbb{Z},+,0,1\right)$, and let $\left(\cM,<\right)$
be an elementary extension of $\left(\cZ,<\right)$ with domain $M$.
Let $N:=\left\{ a\in M\,:\,a\ge0\right\} $. If $\cM$ defines an
infinite subset of $N$, then $\mathcal{Z}$ is interdefinable with
$\left(\mathbb{Z},+,0,1,<\right)$.
\end{cor}

\begin{proof}
Let $L$ be the language of $\cZ$, and let $L_{<}:=L\cup\left\{ <\right\} $.
Suppose that there exist an $L$-formula $\phi(x,y)$ and $b\in M$
such that $\phi(M,b)$ is an infinite subset of $N$.

For each $n\in\N$, let $\psi_{n}(y)$ be the formula $\exists^{\ge n}x\phi(x,y)$.
So $\cM\models\psi_{n}(b)$ for all $n$. Let $\theta(y)$ be the
$L_{<}$-formula $\forall x(\phi(x,y)\rightarrow x\ge0)$. So $\left(\cM,<\right)\models\theta(b)$.
By elementarity, for each $n$ there exists $c_{n}\in\Z$ such that
$\left(\cZ,<\right)\models\psi_{n}(c_{n})\wedge\theta(c_{n})$, so
$\phi(\Z,c_{n})$ is a subset of $\N$ of size at least $n$.

If for some $n$ the set $\phi(\Z,c_{n})$ is infinite, then by \cref{thm:new-dp-minimal-infinite-set-of-natural-numbers},
$\mathcal{Z}$ is interdefinable with $\left(\mathbb{Z},+,0,1,<\right)$.
Otherwise, $\mathcal{Z}$ does not eliminate $\exists^{\infty}$,
hence by \cref{thm:dp-minimal-which-does-not-eliminate-exists-infty},
$\mathcal{Z}$ is interdefinable with $\left(\mathbb{Z},+,0,1,<\right)$.
\end{proof}
It is interesting to ask what can be said for general finitely generated
abelian groups. \cref{subsec:The-case-of-large-doubling}
is already formulated in this generality, but it is not even clear
what should replace \cref{thm:new-dp-minimal-infinite-set-of-natural-numbers}
and \cref{thm:dp-minimal-which-does-not-eliminate-exists-infty}.

For example, for $d\ge3$, the group $\left(\Z^{d},+,0\right)$ has
$2^{\aleph_{0}}$ many dp-minimal expansions that do not eliminate
$\exists^{\infty}$. For simplicity, consider $d=3$. By \cite[Proposition 5.1]{JSW17},
for every group order $\triangleleft$ on $\left(\Z^{3},+,0\right)$,
the structure $\left(\Z^{3},+,0,\triangleleft\right)$ is dp-minimal.
For each $\alpha\in\R\backslash\Q$, let $<_{\alpha}$ be the order
on $\Z^{2}$ defined by $\left(x_{1},x_{2}\right)<_{\alpha}\left(y_{1},y_{2}\right)\iff x_{1}+\alpha x_{2}<y_{1}+\alpha y_{2}$.
Let $\cZ_{\alpha}:=\left(\Z^{3},+,0,\triangleleft_{\alpha}\right)$
be the lexicographic product of $\left(\Z^{2},+,0,<_{\alpha}\right)$
and $\left(\mathbb{Z},+,0,<\right)$, that is, $\left(x_{1},x_{2},x_{3}\right)\triangleleft_{\alpha}\left(y_{1},y_{2},y_{3}\right)$
if and only if $\left(x_{1},x_{2}\right)<_{\alpha}\left(y_{1},y_{2}\right)$
or $\left(x_{1},x_{2}\right)=\left(y_{1},y_{2}\right)$ and $x_{3}<y_{3}$.
So $\cZ_{\alpha}$ is a dp-minimal expansion of $\left(\Z^{3},+,0\right)$
which does not eliminate $\exists^{\infty}$. For each $\alpha\in\R\backslash\Q$,
since $\cZ_{\alpha}$ is a countable structure in a countable language,
there are at most countably many $\beta\in\R\backslash\Q$ for which
$\triangleleft_{\beta}$ is definable in $\cZ_{\alpha}$. Therefore,
the family of structures $\left\{ \cZ_{\alpha}\right\} _{\alpha\in\R\backslash\Q}$
contains a subset of size $2^{\aleph_{0}}$ such that no two distinct
members of this subset are interdefinable.

So what may be true for all finitely generated abelian groups? A natural
conjecture is the following:
\begin{conjecture}
\label{conj:conj_on_fin_gen_abelian_groups}Let $\left(G,+,0\right)$
be a finitely generated abelian group, and let $\mathcal{G}$ be a
dp-minimal expansion of $\left(G,+,0\right)$ which does not eliminate
$\exists^{\infty}$. Then there is a definable subgroup $H\le G$
and there is a definable group order $\triangleleft$ on $H$ such
that $\left(H,+,0,\triangleleft\right)\cong\left(\mathbb{Z},+,0,<\right)$.
\end{conjecture}

Consider the structure $\left(\Z^{2},+,0,<_{\text{lex}}\right)$,
where $<_{\text{lex}}$ is the lexicographic order on $\Z^{2}$. By
\cite[Proposition 5.1]{JSW17}, this structure is dp-minimal. Let
$P:=\left\{ 0\right\} \times\Z$. Then $P$ is externally definable
in $\left(\Z^{2},+,0,<_{\text{lex}}\right)$ (in fact, by \cite[Lemma 2.1]{CH11},
$P$ is definable\footnote{Explicitly, $P=\left\{ x\in\Z^{2}\,:\,\left(1,0\right)\notin\left[0,2\left\vert x\right\vert \right]+2\Z^{2}\right\} ,$
where $\left\vert x\right\vert $ is the absolute value of $x$ with
respect to $<_{\text{lex}}$.}). Let $<_{P}$ be the restriction of $<_{\text{lex}}$ to $P$, i.e.,
$\left(x_{1},x_{2}\right)<_{P}\left(y_{1},y_{2}\right)$ if and only
if $x_{1}=y_{1}=0$ and $x_{2}<y_{2}$. Then $<_{P}$ is externally
definable (in fact, definable) in $\left(\Z^{2},+,0,<_{\text{lex}}\right)$.
Therefore, the structure $\left(\Z^{2},+,0,P,<_{P}\right)$ is dp-minimal.
So it seems that \cref{conj:conj_on_fin_gen_abelian_groups} cannot
be strengthened to say, for example, that there is a group order on
$\left(G,+,0\right)$ which is definable in $\mathcal{G}$.

\section{Dp-rank}

We recall the definitions of $\kappa_{ict}$, $\kappa_{inp}$, dp-rank,
and the Shelah expansion, and some basic facts about them. We also
prove that an apparently weaker condition is in fact sufficient for
$\kappa_{ict}>\kappa$.
\begin{defn}
Let $T$ be a theory, and let $\kappa$ be a cardinal. An \emph{ict-pattern
of depth $\kappa$ }consists of:

\begin{itemize}
\item a collection of formulas $(\phi_{\alpha}(x;y_{\alpha})\,:\,\alpha<\kappa)$,
with $\left\vert x\right\vert =1$, and
\item an array $(b_{\alpha,i}\,:\,i<\omega,\,\alpha<\kappa)$ of tuples
in some model $\cM$ of $T$, with $|b_{\alpha,i}|=|y_{\alpha}|$
\end{itemize}
such that for every $\eta:\kappa\to\omega$ there exists an element
$a_{\eta}\in\cM$ such that 
\[
\cM\models\phi_{\alpha}(a_{\eta};b_{\alpha,i})\iff\eta(\alpha)=i
\]

We define $\kappa_{ict}=\kappa_{ict}(T)$ as the minimal $\kappa$
such that there does not exist an ict-pattern of depth $\kappa$,
and define $\kappa_{ict}=\infty$ if there is no such $\kappa$. 
\end{defn}

\begin{defn}
Let $T$ be a theory, and let $\kappa$ be a cardinal. We say that
$\mbox{dp-rank}(T)<\kappa$ if $\kappa_{ict}(T)\leq\kappa$. We say
that $\mbox{dp-rank}(T)=\kappa$ if $\mbox{dp-rank}(T)<\kappa^{+}$
but $\mbox{dp-rank}(T)\nless\kappa$. We also say that $\mbox{dp-rank}(T)=\infty$
if $\kappa_{ict}(T)=\infty$. $T$ is called \emph{strongly-dependent}
if $\mbox{dp-rank}(T)<\omega$, and is called \emph{dp-minimal} if
$\mbox{dp-rank}(T)=1$. For a structure $\cM$, we let $\kappa_{ict}(\cM):=\kappa_{ict}(Th(\cM))$,
$\mbox{dp-rank}(\cM):=\mbox{dp-rank}(Th(\cM))$, and we say that $\cM$
is \emph{strongly-dependent} or that $\cM$ is \emph{dp-minimal} if
$Th(\cM)$ is, respectively.
\end{defn}

This notion of dp-rank is more commonly used than $\kappa_{ict}$.
It has a small aesthetic advantage over $\kappa_{ict}$ in that, for
theories with infinite models, the smallest possible value of the
dp-rank is $1$, whereas the smallest possible value of $\kappa_{ict}$
is $2$. However, it has the disadvantage that the dp-rank is not
always a well-defined cardinal. For example, if $\kappa_{ict}(T)=\omega$
then $\mbox{dp-rank}(T)<\omega$ but $\mbox{dp-rank}(T)\neq n$ for
all $n<\omega$. The same happens when $\kappa_{ict}$ is a limit
cardinal.

The notion of dp-rank has an equivalent definition in terms of indiscernible
sequences, see \cite[Proposition 4.22]{Simon_2015}.
\begin{fact}[{\cite[Observation 4.13]{Simon_2015}}]
\label{fact:NIP_equivalent_to_non_infinite_dp_rank}A theory $T$
is NIP if and only if $\mbox{dp-rank}(T)<\infty$.
\end{fact}

\begin{defn}
Let $T$ be a theory, and let $\kappa$ be a cardinal. An \emph{inp-pattern
of depth $\kappa$ }consists of:

\begin{itemize}
\item a collection of formulas $(\phi_{\alpha}(x;y_{\alpha})\,:\,\alpha<\kappa)$,
with $\left\vert x\right\vert =1$, and
\item an array $(b_{\alpha,i}\,:\,i<\omega,\,\alpha<\kappa)$ of tuples
in some model of $T$, with $|b_{\alpha,i}|=|y_{\alpha}|$
\end{itemize}
such that: 
\begin{enumerate}
\item for each $\alpha<\kappa$ there exists $k_{\alpha}\in\N$ for which
the row $\left\{ \phi_{\alpha}(x;b_{\alpha,i})\,:\,i<\omega\right\} $
is $k_{\alpha}$-inconsistent, and
\item for every $\eta:\kappa\to\omega$ the path $\left\{ \phi_{\alpha}(x;b_{\alpha,\eta(\alpha)})\,:\,\alpha<\kappa\right\} $
is consistent.
\end{enumerate}
We define $\kappa_{inp}=\kappa_{inp}(T)$ as the minimal $\kappa$
such that there does not exist an inp-pattern of depth $\kappa$,
and define $\kappa_{inp}=\infty$ if there is no such $\kappa$. 
\end{defn}

\begin{rem}
\label{rem:for_inp_ptrn_with_indiscernibles_enough_first_column}In
the definition of an inp-pattern, if the rows of the array $(b_{\alpha,i}\,:\,i<\omega,\,\alpha<\kappa)$
are mutually indiscernible, then instead of requiring that every path
$\left\{ \phi_{\alpha}(x;b_{\alpha,\eta(\alpha)})\,:\,\alpha<\kappa\right\} $
is consistent, it is enough to only require that the first column
$\left\{ \phi_{\alpha}(x;b_{\alpha,0})\,:\,\alpha<\kappa\right\} $
is consistent.
\end{rem}

\begin{fact}[{\cite[Proposition 10]{Adl}}]
\label{fact:k_inp_eq_k_ict_under_nip}For every theory $T$, $\kappa_{inp}\le\kappa_{ict}$,
and if $T$ is NIP, $\kappa_{inp}=\kappa_{ict}$.
\end{fact}

\begin{lem}
\label{lem:ict_pattern_with_bounded_errors_is_enough}Let $T$ be
a theory, let $\kappa$ be a cardinal, and let $C\in\N$. Suppose
that the following pattern exists:

\begin{itemize}
\item a collection of formulas $(\phi_{\alpha}(x;y_{\alpha})\,:\,\alpha<\kappa)$,
with $\left\vert x\right\vert =1$, and
\item an array $(b_{\alpha,i}\,:\,i<\omega,\,\alpha<\kappa)$ of tuples
in some model $\cM$ of $T$, with $|b_{\alpha,i}|=|y_{\alpha}|$
\end{itemize}
such that for every $\eta:\kappa\to\omega$ there exists an element
$a_{\eta}\in\cM$ such that for all $\alpha<\kappa$, $\cM\models\phi_{\alpha}(a_{\eta};b_{\alpha,\eta(\alpha)})$
and $\left\vert \left\{ i<\omega\,:\,\cM\models\phi_{\alpha}(a_{\eta};b_{\alpha,i})\right\} \right\vert \le C$.
Then there exists an ict-pattern of depth $\kappa$.
\end{lem}

\begin{proof}
We may assume that $T$ is NIP, as otherwise, by \cref{fact:NIP_equivalent_to_non_infinite_dp_rank},
$\kappa_{ict}(T)=\infty$. Let $(c_{\alpha,i}\,:\,i<\omega,\,\alpha<\kappa)$
be a mutually indiscernible array locally based on $(b_{\alpha,i}\,:\,i<\omega,\,\alpha<\kappa)$
(see \cite[Lemma 4.2]{Simon_2015}).

Suppose that the set of formulas $\left\{ \phi_{\alpha}(x;c_{\alpha,0})\right\} _{\alpha<\kappa}\cup\left\{ \neg\phi_{\alpha}(x;c_{\alpha,1})\right\} _{\alpha<\kappa}$
is consistent. Note that the rows of the array $((c_{\alpha,2i},c_{\alpha,2i+1})\,:\,i<\omega,\,\alpha<\kappa)$
are mutually indiscernible, and consider the pattern 
\[
\left\{ \phi_{\alpha}(x;c_{\alpha,2i})\wedge\neg\phi_{\alpha}(x;c_{\alpha,2i+1})\,:\,\alpha<\kappa,\,i<\omega\right\} 
\]
Since $T$ is NIP, by indiscernibility, for each $\alpha$ the row
$\left\{ \phi_{\alpha}(x;c_{\alpha,2i})\wedge\neg\phi_{\alpha}(x;c_{\alpha,2i+1})\,:\,i<\omega\right\} $
is inconsistent (see \cite[Lemma 2.7]{Simon_2015}), and hence $k_{\alpha}$-inconsistent
for some $k_{\alpha}\in\N$. So by \cref{rem:for_inp_ptrn_with_indiscernibles_enough_first_column},
this pattern is an inp-pattern of depth $\kappa$. By \cref{fact:k_inp_eq_k_ict_under_nip},
$\kappa_{ict}>\kappa$.

Now suppose that the set of formulas $\left\{ \phi_{\alpha}(x;c_{\alpha,0})\right\} _{\alpha<\kappa}\cup\left\{ \neg\phi_{\alpha}(x;c_{\alpha,1})\right\} _{\alpha<\kappa}$
is inconsistent. By compactness, there is a finite set $I\subseteq\kappa$
such that $\left\{ \phi_{\alpha}(x;c_{\alpha,0})\right\} _{\alpha\in I}\cup\left\{ \neg\phi_{\alpha}(x;c_{\alpha,1})\right\} _{\alpha\in I}$
is inconsistent, so $\forall x(\bigwedge_{\alpha\in I}\phi_{\alpha}(x;c_{\alpha,0})\rightarrow\bigvee_{\alpha\in I}\phi_{\alpha}(x;c_{\alpha,1}))$.
By the mutual indiscernibility of $(c_{\alpha,i}\,:\,i<\omega,\,\alpha<\kappa)$,
for all $i\ge1$ also $\forall x(\bigwedge_{\alpha\in I}\phi_{\alpha}(x;c_{\alpha,0})\rightarrow\bigvee_{\alpha\in I}\phi_{\alpha}(x;c_{\alpha,i}))$.
Let $L=\left\vert I\right\vert $. So in particular, 
\[
\forall x(\bigwedge_{\alpha\in I}\phi_{\alpha}(x;c_{\alpha,0})\rightarrow\bigwedge_{i=1}^{LC+1}\bigvee_{\alpha\in I}\phi_{\alpha}(x;c_{\alpha,i}))
\]
By the choice of the array $(c_{\alpha,i}\,:\,i<\omega,\,\alpha<\kappa)$,
there is a witness for this in the original array $(b_{\alpha,i}\,:\,i<\omega,\,\alpha<\kappa)$,
i.e., if we write $I=\{\alpha_{1},\dots,\alpha_{L}\}$, then there
are elements $\left\{ b_{\alpha_{s},i_{s,t}}\,:\,1\le s\le L,\,0\le t\le LC+1\right\} $
such that 
\[
\forall x(\bigwedge_{s=1}^{L}\phi_{\alpha_{s}}(x;b_{\alpha_{s},i_{s,0}})\rightarrow\bigwedge_{t=1}^{LC+1}\bigvee_{s=1}^{L}\phi_{\alpha_{s}}(x;b_{\alpha_{s},i_{s,t}}))
\]
 But by the assumption of the lemma, there exists an element $a\in\cM$
such that for all $1\le s\le L$, $\cM\models\phi_{\alpha_{s}}(a;b_{\alpha_{s},i_{s,0}})$
and $\left\vert \left\{ i<\omega\,:\,\cM\models\phi_{\alpha_{s}}(a;b_{\alpha_{s},i})\right\} \right\vert \le C$.
By the Pigeonhole Principle, this is a contradiction.
\end{proof}
\begin{defn}
Let $\cM$ be a structure in a language $L$, and fix an elementary
extension $\cN$ of $\cM$ which is $\left\vert \cM\right\vert ^{+}$-saturated. 
\begin{enumerate}
\item An \emph{externally definable} subset of $\cM$ is a subset of $\cM^{k}$
of the form 
\[
\phi(\cM;b):=\left\{ a\in\cM\,:\,\cN\models\phi(a;b)\right\} 
\]
for some $k\ge1$, $\phi(x;y)\in L$, and $b\in\cN^{\left\vert y\right\vert }$. 
\item The \emph{Shelah expansion} of $\cM$, denoted by $\cM^{Sh}$, is
defined to be the structure in the language $L^{Sh}:=\left\{ R_{\phi(x;b)}(x)\,:\,\phi(x;y)\in L,\,b\in\cN^{\left\vert y\right\vert }\right\} $
whose universe is $\cM$ and where each $R_{\phi(x;b)}(x)$ is interpreted
as $\phi(\cM;b)$.
\end{enumerate}
Note that the property of being an externally definable subset of
$\cM$ does not depend on the choice of $\cN$. Hence, up to interdefinability,
$\cM^{Sh}$ does not depend on the choice of $\cN$ (although formally,
a different $\cN$ gives a different language and so a different structure),
so it makes sense to talk about \emph{the} Shelah expansion.
\end{defn}

\begin{fact}[{Shelah. See \cite[Proposition 3.23]{Simon_2015}}]
\label{fact:nip_gives_shelah_expansion_QE}Let $\cM$ be NIP. Then
$\cM^{Sh}$ eliminates quantifiers.
\end{fact}

It follows that every definable set in $\cM^{Sh}$ is externally definable
in $\cM$, and $\left(\cM^{Sh}\right)^{Sh}$ is interdefinable with
$\cM^{Sh}$.
\begin{cor}
\label{cor:shelah_expansion_preserves_dp_rank}Let $\cM$ be any structure.
Then $\kappa_{ict}(\cM^{Sh})=\kappa_{ict}(\cM)$.
\end{cor}

\begin{proof}
This follows from \cref{fact:nip_gives_shelah_expansion_QE}, exactly
as in the proof of \cite[Observation 3.8]{OU11}.
\end{proof}

\section{A lemma and reduction}

We now prove that \cref{thm:new-dp-minimal-infinite-set-of-natural-numbers}
implies \cref{thm:dp-minimal-which-does-not-eliminate-exists-infty}.
The following is the key lemma, which will also be used again in the
proof of \cref{thm:new-dp-minimal-infinite-set-of-natural-numbers}.
\begin{lem}
\label{lem:dp-minimality-prevents-large-two-sided-gaps}Suppose $\mathcal{Z}$
is an expansion of $\left(\mathbb{Z},+,0,1\right)$ which is dp-minimal
and does not eliminate $\exists^{\infty}$, and let $A$ be definable
in an elementary extension $\cM$ of $\mathcal{Z}$. Then there are
only finitely many elements $a\in A$ such that 
\[
A\cap\left(a+\mathbb{Z}\right)=A\cap\left\{ \dots,a-2,a-1,a,a+1,a+2,\dots\right\} 
\]
 is finite.
\end{lem}

\begin{proof}
Assume otherwise. Let $\phi(x,y)$ be a formula witnessing the failure
of elimination of $\exists^{\infty}$, i.e. $\left\vert x\right\vert =1$
and for every $n\in\mathbb{N}$ there exists $b_{n}\in\mathbb{Z}$
such that $n\leq\left\vert \phi(\mathbb{Z},b_{n})\right\vert <\infty$.
Let $n\in\mathbb{N}$. We will build an ict-pattern of depth $2$
and length $n$. Let $a_{1},...,a_{n}$ be such that for each $i$,
$A\cap\left(a_{i}+\mathbb{Z}\right)$ is finite, and the sets $\left\{ A\cap\left(a_{i}+\mathbb{Z}\right)\,:\,1\leq i\leq n\right\} $
are different (i.e., these $a_{i}$ are in pairwise different cosets
of $\mathbb{Z}$). We may assume that $a_{i}$ is the minimal element
of $A\cap\left(a_{i}+\mathbb{Z}\right)$ (otherwise replace $a_{i}$
with this minimal element), and let $c_{i}$ be the maximal element
of $A\cap\left(a_{i}+\mathbb{Z}\right)$. Let $d_{i}:=c_{i}-a_{i}\in\mathbb{N}$,
and let $L:=\max\left\{ d_{i}\,:\,1\leq i\leq n\right\} +1\in\mathbb{N}$.
So for each $i$, $A\cap\left(a_{i}+\mathbb{Z}\right)\subseteq\left[a_{i},a_{i}+L-1\right]$.

Let $b_{n}\in\mathbb{Z}$ be such that $nL\leq\left\vert \phi(\mathbb{Z},b_{n})\right\vert <\infty$,
and denote $B_{n}:=\phi(\mathbb{Z},b_{n})$. Since $b_{n}\in\mathbb{Z}$
we have that $B_{n}=\phi(\cM,b_{n})$. Let $r_{1}$ be the minimal
element of $B_{n}$, and by recursion, for each $1\leq i\leq n-1$
choose $r_{i+1}$ to be the minimal element in $B_{n}$ such that
$r_{i+1}\geq r_{i}+L$. This $r_{i+1}$ exists because (as can be
proved by induction) there are at most $(i-1)L$ elements of $B_{n}$
below $r_{i}$. 

Consider the pattern
\begin{align*}
 & \left(A+r_{i}\right)_{1\leq i\leq n}\\
 & \left(B_{n}+a_{j}\right)_{1\leq j\leq n}
\end{align*}
Then for each pair $1\le i,j\le n$, $r_{i}+a_{j}\in A+r_{i}$ and
$r_{i}+a_{j}\in B_{n}+a_{j}$. But $r_{i}+a_{j}\notin B_{n}+a_{k}$
for $k\neq j$, because $B_{n}+a_{k}$ and $B_{n}+a_{j}$ are disjoint
(since $a_{j}$ and $a_{k}$ are in pairwise different cosets of $\mathbb{Z}$).
And $r_{i}+a_{j}\notin A+r_{k}$ for $k\neq i$: Suppose $r_{i}+a_{j}\in A+r_{k}$.
Then since $r_{i},r_{k}\in\mathbb{Z}$, $a_{j}+r_{i}-r_{k}\in A\cap\left(a_{j}+\mathbb{Z}\right)$.
Since $A\cap\left(a_{j}+\mathbb{Z}\right)\subseteq\left[a_{j},a_{j}+L-1\right]$,
we get $0\le r_{i}-r_{k}\le L-1$, contradicting the construction
of the $r_{i}$. So this is an ict-pattern of depth $2$ and length
$n$. As $n$ was arbitrary, this contradicts the dp-minimality of
$\mathcal{Z}$.
\end{proof}
\begin{rem}
\label{rem:remarks_on_adding_the_order}Let $\mathcal{Z}$ be an expansion
of $\left(\mathbb{Z},+,0,1\right)$. We note the following two simple
observations:
\begin{enumerate}
\item Let $\left(\cM,<\right)$ be an elementary extension of $\left(\cZ,<\right)$.
Then every nonempty definable subset which is bounded from below (resp.
above) has a minimum (resp. maximum). 
\item \label{enu:elim-exists-infty-gives-infinite-bndd-def-set-in-extension}Let
$\phi(x,y)$ be a formula with $\left\vert x\right\vert =1$ such
that for every $n\in\mathbb{N}$ there exists $b_{n}\in\mathbb{Z}$
such that $n\leq\left\vert \phi(\mathbb{Z},b_{n})\right\vert <\infty$.
Then in some elementary extension $\left(\cM,<\right)$ of $\left(\cZ,<\right)$
there exists $b$ such that $\phi(\cM,b)$ is infinite but bounded
from above and below. 
\end{enumerate}
\end{rem}

\begin{cor}
\label{cor:get-definable-set-infinite-bounded-below}Suppose $\mathcal{Z}$
is an expansion of $\left(\mathbb{Z},+,0,1\right)$ which is dp-minimal
and does not eliminate $\exists^{\infty}$, and let $\left(\cM,<\right)$
be an elementary extension of $\left(\cZ,<\right)$. Let $A$ be definable
in $\cM$ (i.e., without $<$), and suppose that $A$ is infinite
and bounded from below. Then there exists an $a\in A$ such that $A\cap\left(a+\mathbb{Z}\right)$
is infinite and bounded from below by $a$. 
\end{cor}

\begin{proof}
By \cref{lem:dp-minimality-prevents-large-two-sided-gaps}, there are
only finitely many elements $a\in A$ such that $A\cap\left(a+\mathbb{Z}\right)$
is finite. Let $F$ denote the set of all these elements. Then $A\backslash F$
is definable, nonempty, and bounded from below, and hence has a minimum
$a$. Then $A\cap\left(a+\mathbb{Z}\right)$ is infinite, but is bounded
from below by $a$. 
\end{proof}
\begin{cor}
\label{cor:thm1_implies_thm2} \cref{thm:new-dp-minimal-infinite-set-of-natural-numbers}
implies \cref{thm:dp-minimal-which-does-not-eliminate-exists-infty}.
\end{cor}

\begin{proof}
Let $\mathcal{Z}$ be an expansion of $\left(\mathbb{Z},+,0,1\right)$
which is dp-minimal and does not eliminate $\exists^{\infty}$, and
let $\phi(x,y)$ be a formula with $\left\vert x\right\vert =1$ such
that for every $n\in\mathbb{N}$ there exists $b_{n}\in\mathbb{Z}$
such that $n\leq\left\vert \phi(\mathbb{Z},b_{n})\right\vert <\infty$.
By \cref{rem:remarks_on_adding_the_order} (\ref*{enu:elim-exists-infty-gives-infinite-bndd-def-set-in-extension}),
in some elementary extension $\left(\cM,<\right)$ of $\left(\cZ,<\right)$
there exists $b$ such that $A:=\phi(\cM,b)$ is infinite but bounded
from above and below. By \cref{cor:get-definable-set-infinite-bounded-below},
there is an $a\in A$ such that $A\cap\left(a+\mathbb{Z}\right)$
is infinite and bounded from below by $a$. So $A^{\prime}:=\left(A-a\right)\cap\mathbb{Z}$
is infinite and bounded from below by $0$ (i.e., $A^{\prime}\subseteq\mathbb{N}$).
The set $A^{\prime}$ is externally definable in $\cZ$, hence definable
in the Shelah expansion $\mathcal{Z}^{Sh}$. Since $\mathcal{Z}$
is dp-minimal, by \cref{cor:shelah_expansion_preserves_dp_rank} $\mathcal{Z}^{Sh}$
is also dp-minimal. By \cref{thm:new-dp-minimal-infinite-set-of-natural-numbers}
$\mathcal{Z}^{Sh}$ is interdefinable with $\left(\mathbb{Z},+,0,1,<\right)$,
so $\mathcal{Z}$ is a reduct of $\left(\mathbb{Z},+,0,1,<\right)$.
Since $\left(\mathbb{Z},+,0,1\right)$ is stable, $\left(\mathbb{Z},+,0,1\right)^{Sh}$
is interdefinable with $\left(\mathbb{Z},+,0,1\right)$, hence $\mathcal{Z}$
cannot be interdefinable with $\left(\mathbb{Z},+,0,1\right)$. By
\cref{fact:no_intermediate_structures_between_the_group_and_the_order},
$\mathcal{Z}$ is interdefinable with $\left(\mathbb{Z},+,0,1,<\right)$.
\end{proof}

\section{Proof of Theorem \texorpdfstring{\ref{thm:new-dp-minimal-infinite-set-of-natural-numbers}}{\ref{thm:new-dp-minimal-infinite-set-of-natural-numbers}}}

In this section we prove \cref{thm:new-dp-minimal-infinite-set-of-natural-numbers}.
The proof proceeds by considering several cases depending on properties
of $A$ or its definable subsets. We treat each case separately, sometimes
getting stronger results.

\subsection{The syndetic case}

This is the easiest case. Here we get that the order is definable.
\begin{defn}
A set $A\subseteq\N$ is called \emph{syndetic} (more precisely, \emph{syndetic
in $\N$}) if there is a finite set $F\subseteq\N$ such that $\N\subseteq A-F:=\bigcup_{n\in F}\left(A-n\right)$. 
\end{defn}

\begin{prop}
\label{prop:syndetic_defines_the_order}Let $\mathcal{Z}$ be an expansion
of $\left(\mathbb{Z},+,0,1\right)$, and suppose that $A\subseteq\N$
is definable in $\mathcal{Z}$ and is syndetic. Then $\mathcal{Z}$
defines $\N$, and hence defines the order $<$. 
\end{prop}

\begin{proof}
By definition, there is a finite set $F\subseteq\N$ such that $\N\subseteq\bigcup_{n\in F}\left(A-n\right)$.
Let $m:=\min F$. So $\bigcup_{n\in F}\left(A-n\right)\subseteq\left[-m,\infty\right)$,
hence $\N=\left(\bigcup_{n\in F}\left(A-n\right)\right)\backslash\left[-m,-1\right]$,
and this set is definable in $\mathcal{Z}$.
\end{proof}

\subsection{The case of bounded two-sided gaps}

In this case we actually get IP. Unlike the other cases, here we make
assumptions about the combinatorial properties of all the infinite
definable subsets of $A$, not only of $A$ itself.
\begin{defn}
Let $A\subseteq\Z$. We say that $N\in\N$ is a \emph{bound on the
two-sided gaps of $A$} if for every $x\in A$ there exists $d\in\left[-N,-1\right]\cup\left[1,N\right]$
such that $x+d\in A$. We say that \emph{$A$ has bounded two-sided
gaps} if there exists a bound $N$ on the two-sided gaps of $A$.
\end{defn}

\begin{defn}
Let $A\subseteq\Z$.
\begin{enumerate}
\item Define a function $L_{A}:\Z\rightarrow\N\cup\{\infty\}$ by
\[
L_{A}(y):=\sup\left\{ m\in\N\,:\,\left[y-m,y-1\right]\cap A=\emptyset\right\} 
\]
\item Define a function $L_{A}:P(\Z)\rightarrow\N\cup\{\infty,-\infty\}$
by
\[
L_{A}(B):=\sup\left\{ L_{A}(y)\,:\,y\in B\text{ and }y>\inf A\right\} 
\]
\end{enumerate}
\end{defn}

\begin{rem}
~
\begin{enumerate}
\item For all $y$, for $m=0$ we have $\left[y-m,y-1\right]=\left[y,y-1\right]=\emptyset$,
so $L_{A}(y)\ge0$. 
\item If $\min A$ exists and $y\le\min A$, then $L_{A}(y)=\infty$. Otherwise,
$L_{A}(y)<\infty$.
\item For all $B\subseteq\Z$, by the above, the set $\left\{ L_{A}(y)\,:\,y\in B\text{ and }y>\inf A\right\} $
is contained in $\N$ (i.e., it does not contain $\infty$). Hence,
if $L_{A}(B)=\infty$ then $B$ must be infinite.
\item If $A\subseteq\N$ is infinite and not syndetic in $\N$ then $L_{A}(A)=\infty$. 
\end{enumerate}
\end{rem}

\begin{lem}
\label{lem:recursion-step-for-bounded-two-sided-gaps}Let $A^{\prime}\subseteq A\subseteq\N$,
and suppose that $L_{A}(A^{\prime})=\infty$ and that $A^{\prime}$
has bounded two-sided gaps. Then there exists $d\in\N$ such that
the set
\[
B_{d}:=\left\{ y\in A^{\prime}\,:\,\left[y-2d,y-1\right]\cap A=\emptyset\text{ and }\left[y+1,y+d\right]\cap A^{\prime}=\{y+d\}\right\} 
\]
has $L_{A}(B_{d})=\infty$.
\end{lem}

\begin{proof}
Let $N$ be a bound on the two-sided gaps of $A^{\prime}$. Let 
\[
C:=\left\{ y\in A^{\prime}\,:\,\left[y-2N,y-1\right]\cap A=\emptyset\right\} =\left\{ y\in A^{\prime}\,:\,2N\le L_{A}(y)\right\} 
\]

Since $L_{A}(A^{\prime})=\infty$, also $L_{A}(C)=\infty$. For each
$d\in\left[1,N\right]$ let 
\[
D_{d}:=\left\{ y\in C\,:\,\left[y+1,y+d\right]\cap A^{\prime}=\{y+d\}\right\} \subseteq B_{d}
\]

By the choice of $N$ and the definition of $C$, for every $y\in C$
there exists $d\in\left[1,N\right]$ such that $y+d\in A^{\prime}$.
Let $d_{y}$ be the first such $d$. Then $y\in D_{d_{y}}$. So $C\subseteq\bigcup_{d=1}^{N}D_{d}$,
and hence there is $d\in\left[1,N\right]$ for which $L_{A}(D_{d})=\infty$,
so $L_{A}(B_{d})=\infty$. 
\end{proof}
\begin{prop}
\label{prop:bounded_two_sided_gaps_gives_IP}Let $A\subseteq\N$ be
infinite and not syndetic, and let $\mathcal{Z}:=\left(\mathbb{Z},+,0,1,A\right)$.
Suppose that every infinite subset of $A$ that is definable in $\mathcal{Z}$
has bounded two-sided gaps. Then the formula $y-x\in A$ has IP. 
\end{prop}

\begin{proof}
Since $A$ is infinite and not syndetic, $L_{A}(A)=\infty$. We define
recursively a decreasing sequence $\left\{ A_{n}\right\} _{n=0}^{\infty}$
of infinite definable subsets of $A$, and a sequence of positive
integers $\left\{ d_{n}\right\} _{n=0}^{\infty}$, such that for all
$n$, $L_{A}(A_{n})=\infty$, and for all $n\ge1$, $d_{n}>2d_{n-1}$.
\\
Let $A_{0}=A$. Suppose that $A_{n}$ has been constructed, and is
an infinite definable subset of $A$ with $L_{A}(A_{n})=\infty$.
By the assumption on $\mathcal{Z}$, $A_{n}$ has bounded two-sided
gaps. By \cref{lem:recursion-step-for-bounded-two-sided-gaps}, there
exists $d\in\N$ such that the set 
\[
B_{n,d}:=\left\{ y\in A_{n}\,:\,\left[y-2d,y-1\right]\cap A=\emptyset\text{ and }\left[y+1,y+d\right]\cap A_{n}=\{y+d\}\right\} 
\]
has $L_{A}(B_{n,d})=\infty$. Let $d_{n}$ be the first such $d$,
and let $A_{n+1}:=B_{n,d_{n}}$. So $A_{n+1}\subseteq A_{n}$ is an
infinite definable subset of $A$ with $L_{A}(A_{n+1})=\infty$. Let
$y\in A_{n+1}$. So $y\in A_{n}$ and also $y+d_{n}\in A_{n}$. So
on one hand, $(y+d_{n})-d_{n}=y\in A_{n}\subseteq A$. On the other
hand, by the definition of $A_{n}$, because $y+d_{n}\in A_{n}$ we
get $\left[(y+d_{n})-2d_{n-1},(y+d_{n})-1\right]\cap A=\emptyset$.
Hence $d_{n}>2d_{n-1}$. Note that it follows that for all $n$, $d_{n}>\sum_{k<n}d_{k}$.

~\\
Let $\psi(x,y)$ be the formula $y-x\in A$. Let $N\ge1$. Let $b\in A_{N+1}$,
and for each $S\subseteq\left[0,N\right]$ define 
\[
b_{S}:=b+\sum_{k\in S}d_{k}
\]

First we show, by backward induction on $0\le m\le N+1$, that if
$S\subseteq\left[m,N\right]$ then $b_{S}\in A_{m}$. For $m=N+1$,
$S=\emptyset$ so $b_{S}=b\in A_{N+1}$. Suppose that the claim is
true for some $m\ge1$. Let $S\subseteq\left[m-1,N\right]$, and let
$S^{\prime}:=S\backslash\{m-1\}\subseteq\left[m,N\right]$. So $b_{S}\in\left\{ b_{S^{\prime}},b_{S^{\prime}}+d_{m-1}\right\} $.
By the induction hypothesis $b_{S^{\prime}}\in A_{m}$, hence, by
the definition of $A_{m}$, $b_{S^{\prime}},b_{S^{\prime}}+d_{m-1}\in A_{m-1}$,
so $b_{S}\in A_{m-1}$.

Now, let $S\subseteq\left[0,N\right]$, and let $m\in\left[0,N\right]$.
If $m\in S$, then for $S^{\prime}:=S\backslash\{m\}$ we have $b_{S}-d_{m}=b_{S^{\prime}}\in A_{0}=A$,
so $\mathcal{Z}\vDash\psi(d_{m},b_{S})$. If $m\notin S$, then for
$S_{1}=S\cap\left[m+1,N\right]$ and $S_{2}=S\cap\left[0,m-1\right]$
we have 
\[
b_{S}-d_{m}=b_{S_{1}}-d_{m}+\sum_{k\in S_{2}}d_{k}
\]
Since $d_{m}>\sum_{k<m}d_{k}$, we have $b_{S_{1}}-d_{m}\le b_{S}-d_{m}\le b_{S_{1}}-d_{m}+\sum_{k<m}d_{k}<b_{S_{1}}$,
so $b_{S}-d_{m}\in\left[b_{S_{1}}-d_{m},b_{S_{1}}-1\right]$. But
since $S_{1}\subseteq\left[m+1,N\right]$ we have $b_{S_{1}}\in A_{m+1}$,
and hence $\left[b_{S_{1}}-2d_{m},b_{S_{1}}-1\right]\cap A=\emptyset$.
Therefore $b_{S}-d_{m}\notin A$, and so, $\mathcal{Z}\nvDash\psi(d_{m},b_{S})$.

In conclusion, we showed that for all $m\in\left[0,N\right]$ and
$S\subseteq\left[0,N\right]$,
\[
\mathcal{Z}\vDash\psi(d_{m},b_{S})\iff m\in S
\]
This is true for all $N$, therefore $\psi(x,y)$ has IP.
\end{proof}

\subsection{\label{subsec:The-case-of-large-doubling}The case of large doubling}

Here we consider what happens when sumsets of $A\cap\left[0,n\right]$
are, asymptotically in $n$, as large as possible relative to $A\cap\left[0,n\right]$.
For our purposes it is enough to consider sums of two elements, but
the same proof works for sums of $k\ge2$ elements, giving $\mbox{dp-rank}\ge k$.

In this subsection, let $\left(G,+,0\right)$ be an abelian group.
\begin{defn}
\label{def:notations_sumsets_popularity}For $k\ge2$, $K\in\N$ and
sets $A_{1},\dots,A_{k}\subseteq G$, denote: 
\begin{enumerate}
\item $A_{1}+\dots+A_{k}:=\left\{ a_{1}+\dots+a_{k}\,:\,a_{i}\in A_{i}\text{ for all }1\le i\le k\right\} $.
\item $k\cdot A:=\underbrace{A+A+\dots+A}_{k\text{ times}}$. Also denote
$1\cdot A:=A$. 
\item $r_{A_{1},\dots,A_{k}}^{k}(x):=\left\vert \{(a_{1},\dots,a_{k})\in A_{1}\times\dots\times A_{k}\,:\,x=a_{1}+\dots+a_{k}\}\right\vert $.
This is the number of ways $x$ can be represented as a sum of $k$
elements, one from each of $\left\{ A_{i}\right\} _{i=1}^{k}$. Note
that the order matters, so, e.g., if $a\neq b$ then $(a,b)$ and
$(b,a)$ are considered to be different representations.
\item $r_{A}^{k}(x):=r_{A,\dots,A}^{k}(x)$.
\item $D_{K}^{k}(A_{1},\dots,A_{k}):=\{x\in A_{1}+\dots+A_{k}\,:\,r_{A_{1},\dots,A_{k}}^{k}(x)\ge K\}$.
\item $D_{K}^{k}(A):=D_{K}^{k}(A,\dots,A)$.
\end{enumerate}
\end{defn}

\begin{rem}
If $\mathcal{G}$ is an expansion of $\left(G,+,0\right)$, and $A_{1},\dots,A_{k}\subseteq G$
are definable in $\mathcal{G}$, then $A_{1}+\dots+A_{k}$ is definable
in $\mathcal{G}$, and for each $K$, $D_{K}^{k}(A_{1},\dots,A_{k})$
is definable in $\mathcal{G}$.
\end{rem}

The following observation is trivial. We state it explicitly to make
its uses clearer.
\begin{obs}
\label{obs:counting_all_reprs_as_sum_of_k}Let $A\subseteq G$ be
a finite set, and let $k\ge2$. Then 
\[
\sum_{x\in k\cdot A}r_{A}^{k}(x)=\left\vert A\right\vert ^{k}
\]
\end{obs}

\begin{prop}
\label{prop:unpopular_k_sumset_gives_dp_rank_atleast_k}Let $\mathcal{G}$
be an expansion of $\left(G,+,0\right)$, and let $A\subseteq G$
be infinite and definable in $\mathcal{G}$. Let $k\ge2$ and $K\in\N$,
and suppose that for all $n\in\N$ there are subsets $B_{n,1},\dots,B_{n,k}\subseteq A$,
each of size at least $n$, such that $B_{n,1}+\dots+B_{n,k}\subseteq\left(k\cdot A\right)\backslash D_{K}^{k}(A)$.
Then $\text{dp-rank}(\mathcal{G})\ge k$.
\end{prop}

\begin{proof}
Consider the following set of formulas:
\begin{align*}
\Phi_{1}:= & \left\{ y_{\alpha,i}\neq y_{\alpha,j}\,:\,1\le\alpha\le k,\,i,j<\omega,\,i\neq j\right\} \\
\Phi_{2}:= & \left\{ y_{\alpha,i}\in A\,:\,1\le\alpha\le k,\,i<\omega\right\} \\
\Phi_{3}:= & \left\{ y_{1,i_{1}}+\dots+y_{k,i_{k}}\notin D_{K}^{k}(A)\,:\,i_{1},\dots,i_{k}\in\omega\right\} \\
\Phi:= & \Phi_{1}\cup\Phi_{2}\cup\Phi_{3}
\end{align*}

By the assumption, $\Phi$ is finitely satisfiable. By compactness,
there is an elementary extension $\mathcal{G}^{*}$ of $\mathcal{G}$
in which this set is realized. Let $\left(b_{\alpha,i}\,:\,1\le\alpha\le k,\,i<\omega\right)$
be a realization, and let $A^{*}$ denote the interpretation inside
$\mathcal{G}^{*}$ of the formula defining $A$ in $\mathcal{G}$.
Consider the pattern:

\[
\left\{ x-b_{\alpha,i}\in(k-1)\cdot A^{*}\,:\,1\le\alpha\le k,\,i<\omega\right\} \text{,}
\]

and for each $\eta:\left\{ 1,\dots,k\right\} \rightarrow\omega$,
let $a_{\eta}:=b_{1,\eta(1)}+\dots+b_{k,\eta(k)}$. Clearly, for each
$1\le\alpha\le k$, 
\[
a_{\eta}-b_{\alpha,\eta(\alpha)}=b_{1,\eta(1)}+\dots+b_{\alpha-1,\eta(\alpha-1)}+b_{\alpha+1,\eta(\alpha+1)}+\dots+b_{k,\eta(k)}\in(k-1)\cdot A^{*}
\]
For each $b\in A^{*}$ such that $a_{\eta}-b\in(k-1)\cdot A^{*}$,
fix a representation $a_{\eta}-b=c_{\eta,b,1}+\dots+c_{\eta,b,k-1}$
with $c_{\eta,b,1},\dots,c_{\eta,b,k-1}\in A^{*}$. So if $b_{1},b_{2}\in A^{*}$,
$b_{1}\neq b_{2}$, and $a_{\eta}-b_{1},a_{\eta}-b_{2}\in(k-1)\cdot A^{*}$,
then $a_{\eta}=c_{\eta,b_{1},1}+\dots+c_{\eta,b_{1},k-1}+b_{1}$ and
$a_{\eta}=c_{\eta,b_{2},1}+\dots+c_{\eta,b_{2},k-1}+b_{2}$ are two
different representations of $a_{\eta}$ as a sum of $k$ elements
from $A^{*}$ (where two representations, $c_{1}+\dots+c_{k}$ and
$c_{1}^{\prime}+\dots+c_{k}^{\prime}$ are considered as equal if
and only if the tuples $\left(c_{1},\dots,c_{k}\right)$ and $\left(c_{1}^{\prime},\dots,c_{k}^{\prime}\right)$
are equal). Hence $r_{A^{*}}^{k}(a_{\eta})\ge\left\vert \left\{ b\in A^{*}\,:\,a_{\eta}-b\in(k-1)\cdot A^{*}\right\} \right\vert $.
On the other hand, $a_{\eta}=b_{1,\eta(1)}+\dots+b_{k,\eta(k)}\notin D_{K}^{k}(A^{*})$,
hence $r_{A^{*}}^{k}(a_{\eta})<K$. Therefore, for each $1\le\alpha\le k$,
we have
\begin{align*}
 & \left\vert \left\{ i<\omega\,:\,a_{\eta}-b_{\alpha,i}\in(k-1)\cdot A^{*}\right\} \right\vert \le\\
 & \le\left\vert \left\{ b\in A^{*}\,:\,a_{\eta}-b\in(k-1)\cdot A^{*}\right\} \right\vert <K
\end{align*}
So the pattern $\left\{ x-b_{\alpha,i}\in(k-1)\cdot A^{*}\,:\,1\le\alpha\le k,\,i<\omega\right\} $
satisfies the assumption of \cref{lem:ict_pattern_with_bounded_errors_is_enough},
therefore $\kappa_{ict}(\mathcal{G})>k$, and so $\mbox{dp-rank}(\mathcal{G})\ge k$. 
\end{proof}
For the rest of this subsection, suppose that $\left(G,+,0\right)$
is finitely generated. So we may assume that $G=\Z^{d}\times F$ where
$1\le d\in\N$ and $\left(F,+,0\right)$ is a finite abelian group.

For an element $a\in G$, write $a=\left(a^{1},a^{2},\dots,a^{d},a^{F}\right)$
where $a^{1},\dots,a^{d}\in\Z$ and $a^{F}\in F$, and define $\left\Vert a\right\Vert :=\max\left\{ \left\vert a^{1}\right\vert ,\dots,\left\vert a^{d}\right\vert \right\} $.
We denote $G^{+}:=\N^{d}\times F$, and for a set $A\subseteq G^{+}$
and for $n\in\N$, we denote $A_{<n}:=\left\{ a\in A\,:\,\left\Vert a\right\Vert <n\right\} $.
In particular, $G_{<n}^{+}:=\left\{ a\in G^{+}\,:\,\left\Vert a\right\Vert <n\right\} $.
\begin{defn}
Let $k\ge2$ and $0<c\le1$.
\begin{enumerate}
\item For a finite set $A\subseteq G^{+}$, we say that \emph{$A$ has $c$-large
$k$-tupling} if $\left\vert k\cdot A\right\vert \ge c\left\vert A\right\vert ^{k}$. 
\item For an infinite set $A\subseteq G^{+}$, we say that \emph{$A$ has
$c$-large lower-asymptotic $k$-tupling} if
\[
\liminf_{n\rightarrow\infty}\frac{\left\vert k\cdot A_{<n}\right\vert }{\left\vert A_{<n}\right\vert ^{k}}\ge c
\]
Equivalently, if for all $\epsilon>0$ there exists $n_{0}$ such
that for all $n\ge n_{0}$, $\left\vert k\cdot A_{<n}\right\vert \ge(c-\epsilon)\left\vert A_{<n}\right\vert ^{k}$.
We say that \emph{$A$ has large lower-asymptotic $k$-tupling} if
the above $\liminf$ is positive, i.e., if there exists $c>0$ for
which $A$ has $c$-large lower-asymptotic $k$-tupling.
\end{enumerate}
For $k=2$ we say ``doubling'' instead of ``$2$-tupling''.
\end{defn}

\begin{defn}
Let $2\le k\in\N$. A \emph{$k$-uniform hypergraph}, also called
a \emph{$k$-graph}, is a pair $\left(V,E\right)$, where $V$ is
a set and $E\subseteq\left[V\right]^{k}:=\left\{ S\subseteq V\,:\,\left\vert S\right\vert =k\right\} $.
A \emph{$k$-partite $k$-graph} is a $k$-graph $\left(V,E\right)$
together with a choice of partition $V=V_{0}\sqcup\dots\sqcup V_{k-1}$,
such that for each $e\in E$ and each $0\le i\le k-1$, $\left\vert e\cap V_{i}\right\vert =1$.
The sets $\left\{ V_{i}\right\} _{i=0}^{k-1}$ are called the \emph{parts}
of the $k$-graph. For $t\in\N$ we denote by $K_{t:k}$ the \emph{complete
$k$-partite $k$-graph with parts of size $t$}, i.e., the $k$-graph
$\left(V,E\right)$ with partition $V=V_{0}\sqcup\dots\sqcup V_{k-1}$,
where $\left\vert V_{i}\right\vert =t$ for each $i$, $\left\{ a_{0},\dots,a_{k-1}\right\} \in E$
whenever $a_{0}\in V_{0},\dots,a_{k-1}\in V_{k-1}$, and there are
no other edges.
\end{defn}

\begin{fact}[{The K\H{o}v{\'a}ri-S{\'o}s-Tur{\'a}n Theorem for hypergraphs \cite[(4.2)]{Fuer91}}]
\label{fact:Kovari-Sos-Turan}

For every $t\ge2$ and every $k\ge2$ there exists a constant $C=C(t,k)$
such that for all $n\in\N$, if $H$ is a $k$-partite $k$-graph
with $n$ vertices in each part and with more than $Cn^{k-\frac{1}{t^{k-1}}}$
edges, then $H$ contains a copy of $K_{t:k}$ as a subgraph.
\end{fact}

\begin{lem}
\label{lem:passing_from_A_n_to_A_kn}Let $\emptyset\neq A\subseteq G^{+}$,
$2\le k\in\N$, and $\epsilon>0$. Then for all $m\in\N$ there exists
$n\ge m$ such that $\left\vert A_{<kn}\right\vert \le\left(k+\epsilon\right)^{d}\left\vert A_{<n}\right\vert $.
\end{lem}

\begin{proof}
Let $m\in\N$. By increasing $m$, we may assume that $A_{<m}\neq\emptyset$.
For $r\in\N$ denote $n_{r}:=k^{r}m$. Suppose towards a contradiction
that for all $r$, $\left\vert A_{<n_{r+1}}\right\vert =\left\vert A_{<kn_{r}}\right\vert >\left(k+\epsilon\right)^{d}\left\vert A_{<n_{r}}\right\vert $.
Then by induction, for all $r$, $\left\vert A_{<n_{r}}\right\vert >\left(k+\epsilon\right)^{dr}\left\vert A_{<m}\right\vert $. 

Note that
\[
G_{<n_{r}}^{+}=\bigsqcup_{0\le i_{1},\dots,i_{d}\le k^{r}-1}\left(G_{<m}^{+}+\left(i_{1}m,\dots,i_{d}m,0\right)\right)
\]
hence, there are $0\le i_{1},\dots,i_{d}\le k^{r}-1$ such that 
\[
\left\vert A_{<n_{r}}\cap\left(G_{<m}^{+}+\left(i_{1}m,\dots,i_{d}m,0\right)\right)\right\vert \ge\frac{1}{k^{dr}}\left\vert A_{<n_{r}}\right\vert \ge\frac{\left(k+\epsilon\right)^{dr}}{k^{dr}}\left\vert A_{<m}\right\vert =\left(\frac{k+\epsilon}{k}\right)^{dr}\left\vert A_{<m}\right\vert 
\]
On the other hand, 
\[
\left\vert A_{<n_{r}}\cap\left(G_{<m}^{+}+\left(i_{1}m,\dots,i_{d}m,0\right)\right)\right\vert \le\left\vert G_{<m}^{+}\right\vert =m^{d}\left\vert F\right\vert 
\]
So $\left(\frac{k+\epsilon}{k}\right)^{dr}\left\vert A_{<m}\right\vert \le m^{d}\left\vert F\right\vert $.
This is true for all $r$, but $\left(\frac{k+\epsilon}{k}\right)^{dr}\left\vert A_{<m}\right\vert \xrightarrow[r\to\infty]{}\infty$,
a contradiction. 
\end{proof}
\begin{prop}
\label{prop:large_k_tupling_gives_dp_rank_atleast_k}Let $\mathcal{G}$
be an expansion of $\left(G,+,0\right)$, and let $A\subseteq G^{+}$
be infinite and definable in $\mathcal{G}$. Let $k\ge2$ and suppose
that $A$ has large lower-asymptotic $k$-tupling. Then there exists
$K\in\N$ for which the assumptions of \cref{prop:unpopular_k_sumset_gives_dp_rank_atleast_k}
are satisfied, and hence $\text{dp-rank}(\mathcal{G})\ge k$.
\end{prop}

\begin{proof}
By the definition of having large lower-asymptotic $k$-tupling, there
exist $c>0$ and $n_{0}\in\N$ such that for all $n\ge n_{0}$, $\left\vert k\cdot A_{<n}\right\vert \ge\frac{c}{2}\left\vert A_{<n}\right\vert ^{k}$.
Let $t\in\N$. Since $A$ is infinite, there is $m\ge n_{0}$ such
that $\frac{c}{4}\left\vert A_{<m}\right\vert ^{k}>C(t,k)\left\vert A_{<m}\right\vert ^{k-\frac{1}{t^{k-1}}}$,
where $C(t,k)$ is as in \cref{fact:Kovari-Sos-Turan}. By \cref{lem:passing_from_A_n_to_A_kn}
there exists $n\ge m$ such that $\left\vert A_{<kn}\right\vert \le\left(k+1\right)^{d}\left\vert A_{<n}\right\vert $.
Note that still $\frac{c}{4}\left\vert A_{<n}\right\vert ^{k}>C(t,k)\left\vert A_{<n}\right\vert ^{k-\frac{1}{t^{k-1}}}$.

Let $K=\frac{4\left(k+1\right)^{dk}}{c}$. Consider the map $r_{A_{<kn}}^{k}:k\cdot A_{<kn}\rightarrow\N$
and the set $D_{K}^{k}(A_{<kn})$, as defined in \cref{def:notations_sumsets_popularity}.
By Markov's inequality (for the counting measure on $k\cdot A_{<kn}$)
and \cref{obs:counting_all_reprs_as_sum_of_k}, 
\begin{align*}
\left\vert D_{K}^{k}(A_{<kn})\right\vert \le & \frac{1}{K}\sum_{x\in k\cdot A_{<kn}}r_{A_{<kn}}^{k}(x)=\frac{1}{K}\left\vert A_{<kn}\right\vert ^{k}
\end{align*}
So 
\begin{align*}
\left\vert \left(k\cdot A_{<n}\right)\backslash D_{K}^{k}(A_{<kn})\right\vert \ge & \left\vert k\cdot A_{<n}\right\vert -\left\vert D_{K}^{k}(A_{<kn})\right\vert \\
\ge & \left\vert k\cdot A_{<n}\right\vert -\frac{1}{K}\left\vert A_{<kn}\right\vert ^{k}\\
\ge & \frac{c}{2}\left\vert A_{<n}\right\vert ^{k}-\frac{1}{K}\left\vert A_{<kn}\right\vert ^{k}\\
\ge & \frac{c}{2}\left\vert A_{<n}\right\vert ^{k}-\frac{1}{K}\left(\left(k+1\right)^{d}\left\vert A_{<n}\right\vert \right)^{k}\\
= & \frac{c}{2}\left\vert A_{<n}\right\vert ^{k}-\frac{1}{K}\left(k+1\right)^{dk}\left\vert A_{<n}\right\vert ^{k}\\
= & \left(\frac{c}{2}-\frac{\left(k+1\right)^{dk}}{K}\right)\left\vert A_{<n}\right\vert ^{k}\\
= & \frac{c}{4}\left\vert A_{<n}\right\vert ^{k}
\end{align*}
And so
\begin{align*}
 & \left\vert \left\{ (a_{1},\dots,a_{k})\in A_{<n}^{k}\,:\,a_{1}+\dots+a_{k}\in\left(k\cdot A_{<n}\right)\backslash D_{K}^{k}(A_{<kn})\right\} \right\vert \ge\\
 & \ge\left\vert \left(k\cdot A_{<n}\right)\backslash D_{K}^{k}(A_{<kn})\right\vert \ge\frac{c}{4}\left\vert A_{<n}\right\vert ^{k}
\end{align*}

Define a $k$-partite $k$-graph $\left(V,E\right)$ as follows: the
vertices are $V=V_{0}\sqcup\dots\sqcup V_{k-1}$, where for each $i$,
$V_{i}:=\left\{ i\right\} \times A_{<n}$, and the edges are 
\[
E:=\left\{ (a_{1},\dots,a_{k})\in A_{<n}^{k}\,:\,a_{1}+\dots+a_{k}\in\left(k\cdot A_{<n}\right)\backslash D_{K}^{k}(A_{<kn})\right\} 
\]
Since $\left(V,E\right)$ has $\left\vert A_{<n}\right\vert $ vertices
in each part, and at least $\frac{c}{4}\left\vert A_{<n}\right\vert ^{k}>C(t,k)\left\vert A_{<n}\right\vert ^{k-\frac{1}{t^{k-1}}}$
edges, by \cref{fact:Kovari-Sos-Turan} for $t,k$, this graph contains
a copy of $K_{t:k}$ as a subgraph. This means that there are subsets
$B_{t,1},\dots,B_{t,k}\subseteq A_{<n}$, each of size $t$, such
that $B_{t,1}+\dots+B_{t,k}\subseteq\left(k\cdot A_{<n}\right)\backslash D_{K}^{k}(A_{<kn})$. 

Note that if $x\in k\cdot A_{<n}$ then $\left\Vert x\right\Vert <kn$.
So if $a_{1},\dots,a_{k}\in A$ are such that $a_{1}+\dots+a_{k}=x$,
then for each $1\le i\le d$, $a_{1}^{i}+\dots+a_{k}^{i}=x^{i}<kn$.
Since for all $1\le l\le k$, $a_{l}^{i}\ge0$, we get that for each
$1\le l\le k$, $a_{l}^{i}<kn$. Hence $a_{1},\dots,a_{k}\in A_{<kn}$.
Therefore $\left(k\cdot A_{<n}\right)\backslash D_{K}^{k}(A_{<kn})=\left(k\cdot A_{<n}\right)\backslash D_{K}^{k}(A)$.

Since $t$ was arbitrary, by \cref{prop:unpopular_k_sumset_gives_dp_rank_atleast_k},
$\text{dp-rank}(\mathcal{G})\ge k$.
\end{proof}
We now prove an additional lemma that will be used in the next section.
\begin{lem}
\label{lem:no_large_k_tupling_gives_failure_to_eliminate_exists_infty}Let
$\mathcal{G}$ be an expansion of $\left(G,+,0\right)$, and let $A\subseteq G^{+}$
be infinite and definable in $\mathcal{G}$. Let $k\ge2$ and suppose
that $A$ does not have large lower-asymptotic $k$-tupling. Then
$\mathcal{G}$ does not eliminate $\exists^{\infty}$.
\end{lem}

\begin{proof}
By the definition of large lower-asymptotic $k$-tupling, for every
$c>0$ there are infinitely many numbers $n\in\N$ such that $\left\vert k\cdot A_{<n}\right\vert <c\left\vert A_{<n}\right\vert ^{k}$.
First we show that for all $K\in\N$, $D_{K}^{k}(A)\neq\emptyset$,
i.e., there exists $b\in k\cdot A$ such that $r_{A}^{k}(b)\ge K$.
Suppose otherwise, and let $c=\frac{1}{K}$. Let $n\in\N$ be such
that $\left\vert k\cdot A_{<n}\right\vert <c\left\vert A_{<n}\right\vert ^{k}$.
So for all $b\in k\cdot A_{<n}$, $r_{A_{<n}}^{k}(b)\le r_{A}^{k}(b)<K$.
Therefore, by \cref{obs:counting_all_reprs_as_sum_of_k}, 
\[
\left\vert A_{<n}\right\vert ^{k}=\sum_{b\in k\cdot A_{<n}}r_{A_{<n}}^{k}(b)<K\left\vert k\cdot A_{<n}\right\vert <Kc\left\vert A_{<n}\right\vert ^{k}=\left\vert A_{<n}\right\vert ^{k}
\]
a contradiction.

Now, consider the formula $\psi(x,y)$ given by 
\[
x\in A\wedge\exists z_{1},\dots,z_{k-1}\in A\left(y=x+z_{1}+\dots+z_{k-1}\right)
\]
Since $A\subseteq G^{+}$, for all $b\in G$ the set $\psi(G,b)$
is finite. Let $L\in\N$, and let $b\in k\cdot A$ such that $r_{A}^{k}(b)\ge L^{k}$.
Then $r_{A}^{k}(b)\le\left\vert \psi(G,b)\right\vert ^{k}$, and therefore
$\left\vert \psi(G,b)\right\vert \ge L$. This shows that $\mathcal{G}$
does not eliminate $\exists^{\infty}$ for $\psi(x,y)$.
\end{proof}

\subsection{The remaining case}

Here we consider what happens when $A$ does not have bounded two-sided
gaps and also does not have large doubling.
\begin{prop}
\label{prop:dp_min_gives_either_elim_exsts_infty_or_bdd_two_sdd_gaps}Let
$\mathcal{Z}$ be an expansion of $\left(\mathbb{Z},+,0,1\right)$
which does not eliminate $\exists^{\infty}$, and let $A\subseteq\Z$
be infinite and definable in $\mathcal{Z}$. Suppose that $A$ does
not have bounded two-sided gaps. Then $\mathcal{Z}$ is not dp-minimal.
\end{prop}

\begin{proof}
By the definition of bounded two-sided gaps, for every $n\in\N$ there
exists $a_{n}\in A$ such that for all $d\in\left[-n,-1\right]\cup\left[1,n\right]$,
$a_{n}+d\notin A$. Since $A$ is infinite, these elements can be
chosen to be pairwise distinct. By compactness, in some elementary
extension $\mathcal{Z}^{*}$ of $\mathcal{Z}$ there are $\left(b_{n}\right)_{n\in\omega}$
which are pairwise distinct, such that for all $n\in\N$, $A^{*}\cap\left(b_{n}+\mathbb{Z}\right)=\left\{ b_{n}\right\} $,
where $A^{*}$ denotes the interpretation inside $\mathcal{Z^{*}}$
of the formula defining $A$ in $\cZ$. By \cref{lem:dp-minimality-prevents-large-two-sided-gaps},
$\mathcal{Z}$ is not dp-minimal.
\end{proof}
\begin{cor}
\label{cor:dp_min_gives_either_large_doubling_or_bdd_two_sdd_gaps}Let
$\mathcal{Z}$ be an expansion of $\left(\mathbb{Z},+,0,1\right)$,
and let $A\subseteq\N$ be infinite and definable in $\mathcal{Z}$.
Suppose that $A$ does not have bounded two-sided gaps, and suppose
also that $A$ does not have large lower-asymptotic doubling. Then
$\mathcal{Z}$ is not dp-minimal.
\end{cor}

\begin{proof}
This follows from \cref{lem:no_large_k_tupling_gives_failure_to_eliminate_exists_infty}
and \cref{prop:dp_min_gives_either_elim_exsts_infty_or_bdd_two_sdd_gaps}.
\end{proof}
Although in \cref{cor:dp_min_gives_either_large_doubling_or_bdd_two_sdd_gaps}
we could only conclude that $\mathcal{Z}$ is not dp-minimal, we believe
that in fact the following holds:
\begin{conjecture}
Let $\mathcal{Z}$ be an expansion of $\left(\mathbb{Z},+,0,1\right)$,
and let $A\subseteq\N$ be infinite and definable in $\mathcal{Z}$.
Suppose that $A$ does not have bounded two-sided gaps, and suppose
also that for some $k\in\N$, $A$ does not have large lower-asymptotic
$k$-tupling. Then $\text{dp-rank}(\mathcal{Z})\ge n$ for all $n\in\N$.
\end{conjecture}

If this is true, it will follow that \cref{thm:new-dp-minimal-infinite-set-of-natural-numbers}
remains true when, instead of assuming that $\mathcal{Z}$ is dp-minimal,
we only assume that $\text{dp-rank}(\mathcal{Z})=n$ for some $n\in\N$.
Also note that the proof of \cref{cor:thm1_implies_thm2} itself uses
the assumption of dp-minimality (via \cref{lem:dp-minimality-prevents-large-two-sided-gaps}),
and we do not know whether \cref{cor:thm1_implies_thm2} remains true
if in \cref{thm:new-dp-minimal-infinite-set-of-natural-numbers} and
\cref{thm:dp-minimal-which-does-not-eliminate-exists-infty} we only
assume that $\text{dp-rank}(\mathcal{Z})=n$ for some $n\in\N$, instead
of assuming that $\mathcal{Z}$ is dp-minimal. Nevertheless, we believe
that \cref{thm:dp-minimal-which-does-not-eliminate-exists-infty} also
remains true under this weaker assumption.

\subsection{Putting it all together}

We now have all the ingredients for the proof of the main theorem.
\begin{thm}
Let $\mathcal{Z}$ be a dp-minimal expansion of $\left(\mathbb{Z},+,0,1\right)$,
and suppose that there exists an infinite set $A\subseteq\mathbb{N}$
that is definable in $\mathcal{Z}$. Then $\mathcal{Z}$ is interdefinable
with $\left(\mathbb{Z},+,0,1,<\right)$.
\end{thm}

\begin{proof}
Let $A\subseteq\mathbb{N}$ be an infinite set definable in $\mathcal{Z}$.
If $A$ is syndetic, then by \cref{prop:syndetic_defines_the_order},
$\mathcal{Z}$ defines the order $<$. Therefore, by \cref{fact:no_strongly_dependent_expansions_of_the_order},
$\mathcal{Z}$ is interdefinable with $\left(\mathbb{Z},+,0,1,<\right)$.
Suppose that $A$ is not syndetic. If every infinite subset of $A$
that is definable in $\left(\mathbb{Z},+,0,1,A\right)$ has bounded
two-sided gaps, then by \cref{prop:bounded_two_sided_gaps_gives_IP}
we get a contradiction. Otherwise, there exists an infinite subset
$A^{\prime}\subseteq A$ that is definable in $\left(\mathbb{Z},+,0,1,A\right)$
(and hence in $\mathcal{Z}$), and does not have bounded two-sided
gaps. If $A^{\prime}$ does not have large lower-asymptotic doubling,
then by \cref{cor:dp_min_gives_either_large_doubling_or_bdd_two_sdd_gaps}
we get a contradiction. Otherwise, by \cref{prop:large_k_tupling_gives_dp_rank_atleast_k}
we get a contradiction. This completes the proof.
\end{proof}

\subsection*{Acknowledgements:}

I would like to thank Hagai Lavner for many helpful conversations
about additive combinatorics, and my supervisor, Itay Kaplan, for
suggesting the proof of \cref{lem:ict_pattern_with_bounded_errors_is_enough},
and for providing helpful feedback throughout the development of this
paper. I would also like to thank the anonymous referee for their
numerous comments and suggestions. Especially I am grateful for \cref{cor:cor_of_both_thms_together},
and for \cref{conj:conj_on_fin_gen_abelian_groups} and the example
just before it.

\bibliographystyle{plainurl}
\bibliography{on_dp_minimal_expansions_of_the_integers}

\end{document}